\documentclass[4pt]{article} 

\usepackage[utf8]{inputenc} 
\usepackage{geometry} 
\usepackage{graphicx} 

\usepackage[colorlinks,citecolor=red]{hyperref}
\usepackage{amsmath}
\usepackage{amsfonts}
\usepackage{dsfont}
\usepackage{mathrsfs}
\usepackage{latexsym}
\usepackage{graphicx}
\usepackage{amssymb}
\usepackage{float}
\usepackage{epsfig}
\usepackage{epstopdf}
\usepackage{color,xcolor}
\usepackage{booktabs} 
\usepackage{array} 
\usepackage{paralist} 
\usepackage{verbatim} 
\usepackage{subfig} 

\newtheorem{theorem}{Theorem}[section]
\newtheorem{lemma}[theorem]{Lemma}
\newtheorem{corollary}[theorem]{Corollary}

\newtheorem{claim}[theorem]{Claim}

\newtheorem{conjecture}[theorem]{Conjecture}

\usepackage{fancyhdr} 
\usepackage{sectsty}
\allsectionsfont{\sffamily\mdseries\upshape} 

\usepackage[nottoc,notlof,notlot]{tocbibind} 
\usepackage[titles,subfigure]{tocloft} 

\title{Spanning tree packing, edge-connectivity and eigenvalues of graphs with given girth}

\author{Ruifang Liu\thanks{School of Mathematics and Statistics, Zhengzhou
University, Zhengzhou, Henan 450001, China. Email: rfliu@zzu.edu.cn},
Hong-Jian Lai \thanks{Corresponding author. Department of Mathematics, West Virginia
University, Morgantown, WV 26506, USA. E-mail: hjlai@math.wvu.edu},
Yingzhi Tian \thanks{College of Mathematics and System Sciences, Xinjiang University, Urumqi, Xinjiang 830046, China.
Email: tianyzhxj@163.com}}
\date{} 

\begin{document}
\maketitle

\begin{abstract}
\maketitle
Let $\tau(G)$ and $\kappa'(G)$ denote the edge-connectivity
and the spanning tree packing number of a graph $G$, respectively.
Proving a conjecture initiated by Cioaba and Wong, Liu et al. in 2014 showed that
for any simple graph $G$ with minimum degree $\delta \ge 2k \ge 4$, if the second
largest adjacency eigenvalue of $G$ satisfies
$\lambda_2(G) < \delta - \frac{2k-1}{\delta+1}$, then $\tau(G) \ge k$. Similar results
involving the Laplacian eigenvalues and the signless Laplacian eigenvalues of $G$ are also
obtained. In this paper, we find a function $f(\delta, k, g)$ such that for
every graph $G$ with minimum degree $\delta \ge 2k \ge 4$ and girth $g \ge 3$, if its second
largest adjacency eigenvalue  satisfies
$\lambda_2(G) < f(\delta, k, g)$, then $\tau(G) \ge k$. As $f(\delta, k, 3) = \delta - \frac{2k-1}{\delta+1}$,
this extends the above-mentioned result of Liu et al. Related results involving the girth of the graph, Laplacian eigenvalues
and the signless Laplacian eigenvalues to describe $\tau(G)$ and $\kappa'(G)$ are also obtained.


\end{abstract}

\section{Introduction}

We consider finite and simple graphs and follow
\cite{BoMu08} for undefined terms and notation. In particular, $\Delta(G), \delta(G), \kappa'(G)$ and
$\kappa(G)$ denote the maximum degree, the minimum degree, the edge-connectivity and connectivity
of a graph $G$, respectively. The girth of a graph $G$, is defined as
\[
g(G) = \left\{
\begin{array}{ll}
\min\{|E(C)|: \mbox{ $C$ is a cycle of $G\}$ } & \mbox{ if $G$ is not acyclic, }
\\
\infty & \mbox{ if $G$ is acyclic. }
\end{array} \right.
\]
Let $\overline{d}(G)$ be the average degree of $G$,
and $\tau(G)$ be the maximum number of edge-disjoint spanning trees contained in $G$.
A literature review on $\tau(G)$ can be found in \cite{Palm01}.
As in \cite{BoMu08}, for a vertex subset $S\subseteq V(G)$,  $G[S]$ is the
subgraph of $G$ induced by $S$.

Let $G$ be a simple graph of vertex set $\{v_1,\ldots,v_n\}$. The
the adjacency matrix  of $G$ is an $n \times n$ matrix
$A(G)=(a_{uv})$, where $u,v \in V(G)$ and $a_{uv}$ is the number
of edges joining $u$ and $v$ in $G$. As $G$ is simple,
$A(G)$ is symmetric $(0, 1)$-matrix.
Eigenvalues of $G$ are the eigenvalues of $A(G)$. We use
$\lambda_i(G)$ to denote the $i$th largest eigenvalue of $G$. So $\lambda_1(G)\geq \lambda_2(G)\geq
\cdots \geq \lambda_n(G)$.
Let $D(G)$ be the degree diagonal matrix of $G$. The matrices
$L(G) = D(G)-A(G)$ and $Q(G) = D(G) + A(G)$ are the Laplacian matrix and the signless
Laplacian matrix of $G$, respectively. We use $\mu_i(G)$ and $q_i(G)$ to denote the $i$th largest
eigenvalue of $L(G)$ and $Q(G)$, respectively.

Fiedler \cite{Fied73} initiated the investigation between graph connectivity and graph eigenvalues.
Motivated by Kirchhoff's matrix tree theorem \cite{Kirc47}
and by a problem of Seymour (see Reference [19] of \cite{CiWo12}),
Cioab\u{a} and Wong \cite{CiWo12} initiated the following conjecture.

\begin{conjecture} (Cioab\u{a} and Wong \cite{CiWo12},  Gu et al \cite{GLLY16}, Li and Shi \cite{LiSh13}
and Liu et al \cite{LiHL14}) \label{conj1}
Let $k$ be an integer with $k \geq 2$ and $G$ be a graph with minimum degree $\delta \ge 2k$
and maximum degree $\Delta$. If $\lambda_2(G) < \delta-\frac{2k-1}{\delta+1}$, then $\tau(G) \ge  k$.
\end{conjecture}

Several studies made progresses towards Conjecture \ref{conj1}, as seen in
\cite{CiWo12, GLLY16, LiSh13, LiHL14, LHGL14}. The conjecture is finally settled in
\cite{LHGL14}.

\begin{theorem}(Liu, Hong, Gu and Lai \cite{LHGL14})\label{th1}
Let $k \ge 2$ be an integer, and $G$ be a graph with $\delta(G) \geq 2k \ge 4$.
Each of the following holds.
\\
(i) If $\lambda_{2}(G)< \delta(G) -\frac{2k-1}{\delta(G)+1}$, then $\tau(G)\geq k$.
\\
(ii) If $\mu_{n-1}(G) > \frac{2k-1}{\delta(G)+1}$, then $\tau(G)\geq k$.
\\
(iii) If $q_{2}(G)<2\delta(G) -\frac{2k-1}{\delta(G)+1}$, then $\tau(G)\geq k$.
\end{theorem}

Nash-Williams \cite{Nash61} and Tutte \cite{Tutt61} proved a fundamental
theorem on spanning tree packing number of a graph $G$.

\begin{theorem}(Nash-Williams \cite{Nash61} and Tutte \cite{Tutt61}) \label{tree-packing}
Let G be a connected graph and let $k>0$ be an integer. Then $\tau(G)\geq k$ if and only if
for any partition $(V_{1}, \ldots, V_t)$ of $V(G)$, $\sum_{i=1}^t d(V_{i}) \ge 2k(t-1).$
\end{theorem}

As consequences of Theorem \ref{tree-packing}, relationship between
$\tau(G)$ and $\kappa'(G)$ has been investigated, as seen in \cite{Gusf83} and
\cite{Kund74}, among others. A characterization is proved in \cite{CaLS09}.

\begin{theorem} (Catlin, Lai and Shao \cite{CaLS09}) \label{k-t}
Let $k \ge 1$ be an integer. Then $\kappa'(G) \ge 2k$ if and only if
for any subset $X \subseteq E(G)$ with $|X| \le k$, $\tau(G - X) \ge k$.
\end{theorem}

Cioab\u{a} in \cite{Cioa10} initiated the investigation on the relationship between
graph adjacency eigenvalues and edge-connectivity. A number of results have been obtained.

\begin{theorem} \label{edge-conn}
Let $d$ and $k$ be integers with $d \ge k \ge 2$, and let $G$ be a simple graph
on $n$ vertices with $\delta = \delta(G) \ge k$.
\\
(i) (Cioab\u{a} \cite{Cioa10}) If $G$ is $d$-regular and $\lambda_{2}(G)\leq d-\frac{(k-1)n}{(d+1)(n-d-1)},$
then $\kappa'(G)\geq k.$
\\
(ii) (Cioab\u{a} \cite{Cioa10}) If $G$ is $d$-regular and $\lambda_{2}(G)<d-\frac{2(k-1)}{d+1},$
then $\kappa'(G)\geq k.$
\\
(iii) (Gu et al \cite{GLLY16}) If
$\lambda_{2}(G)<\delta-\frac{2(k-1)}{\delta+1},$ then $\kappa'(G)\geq k.$
\\
(iv) (Liu et al \cite{LiHL14}) If
$\lambda_{2}(G)\leq \delta-\frac{(k-1)n}{(\delta+1)(n-\delta-1)},$
then $\kappa'(G)\geq k.$
\end{theorem}

These motivates the current research. It is natural to understand whether
we will have a different range of the eigenvalues to predict the values of
$\tau$ or $\kappa'$, when we are restricted to certain graph families such as bipartite
graphs. The goal of this study is investigate, when the girth of a graph $G$ is known,
the relationship between the eigenvalues of $G$ and $\tau(G)$, as well as $\kappa'(G)$.
Motivated by the methods deployed in \cite{LHGL14}, for any graph $G$ with
adjacency matrix $A$ and diagonal degree matrix $D$, we define $\lambda_i(G,a)$ to be
the $i$th largest eigenvalues of $aD+A$, where $a \ge -1$ is a real number.
For any integers $\delta$ and $g$ with $\delta  > 0$ and $g \ge 3$, define $t = \lfloor \frac{g-1}{2} \rfloor$,
and $n_1^* = n_1^*(\delta, g)$ as follows.
\begin{equation} \label{n1*}
n_{1}^{*}=
\left\{
\begin{array}{lc}
1+\delta+\sum_{i=2}^{t}(\delta-1)^{i}, &\,~~~ \text{\mbox{if}~ $g=2t+1$};\\
2+2(\delta-1)^{t}+\sum_{i=1}^{t-1}(\delta-1)^{i}, &\, ~~~\text{\mbox{if}~ $g=2t+2$}.
\end{array}
\right.
\end{equation}
The main results are the following.

\begin{theorem}\label{main1}
Let $g$ and $k$ be integers with $g \ge 3$ and $k \ge 2$, $a \ge -1$ be a real number,
and $G$ be a simple graph of order $n$ with minimum degree $\delta \geq k \geq 2$ and girth $g$.
Each of the following holds.
\\
(i) If $\displaystyle \lambda_{2}(G, a)\leq (a+1)\delta-\frac{(k-1)n}{n_{1}^{*}(n-n_{1}^{*})}$,
then $\kappa'(G)\geq k$.
\\
(ii) If
$\displaystyle
\lambda_{2}(G, a)<(a+1)\delta-\frac{2(k-1)}{n_{1}^{*}}$,
then $\kappa'(G)\geq k$.
\end{theorem}

\begin{theorem}\label{main2}
Let $g$ and $k$ be integers with $g \ge 3$ and $k \ge 2$, $a \ge -1$ be a real number,
and $G$ be a simple graph of order $n$ with minimum degree $\delta \geq 2k \geq 4$ and girth $g$.
If $\displaystyle \lambda_{2}(G, a)< (a+1)\delta-\frac{2k-1}{n_{1}^{*}}$, then $\tau(G)\geq k$.
\end{theorem}

When we choose $a \in \{0, 1, -1\}$, then Theorems \ref{main1} and \ref{main2} will lead to results
using $\lambda_{2}(G)$,  $\mu_{n-1}(G)$ and $q_{2}(G)$ to describe $\kappa'(G)$ and $\tau(G)$.
In particular, Theorem \ref{main2} has the following corollary. As $n_1^*(\delta, 3) = \delta + 1$,
Corollary \ref{cor2} extends Theorem \ref{th1}.

\begin{corollary}\label{cor2}
Let $g$ and $k$ be integers with $g \ge 3$ and $k \ge 2$,
and $G$ be a simple graph of order $n$ with minimum degree $\delta \geq 2k \geq 4$ and girth $g$.
Each of the following holds.
\\
(i) If $\lambda_{2}(G)< \delta-\frac{2k-1}{n_{1}^{*}}$, then $\tau(G)\geq k$.
\\
(ii) If $\mu_{n-1}(G)> \frac{2k-1}{n_{1}^{*}}$, then $\tau(G)\geq k$.
\\
(iii) If $q_{2}(G)< 2\delta-\frac{2k-1}{n_{1}^{*}}$, then $\tau(G)\geq k$.
\end{corollary}

The arguments adopted in this paper are refinements and improvements of those
presented in \cite{LiHL14} and \cite{LHGL14}.
In the next section, we present the interlacing technique, a common tool
in spectral theory of matrices. The proofs of the main results are in the subsequent
sections.

\section{Preliminaries}

The main tool in our paper is the eigenvalue interlacing technique described below.

Given two non-increasing real sequences
$\theta_{1}\geq \theta_{2}\geq \cdots \geq \theta_{n}$ and
$\eta_{1}\geq \eta_{2}\geq \cdots \geq \eta_{m}$
with $n>m,$ the second sequence is said to {\em interlace}
the first one if $\theta_{i}\geq \eta_{i}\geq\theta_{n-m+i}$
for $i=1, 2, \ldots, m.$
The interlacing is {\em tight} if exists an integer $k\in[0, m]$
such that $\theta_{i}=\eta_{i}$ for $1\leq i\leq k$ and $\theta_{n-m+i}=\eta_{i}$ for
$k+1\leq i\leq m.$

\begin{lemma}(Cauchy Interlacing \cite{BrHa12})\label{le2.1}
Let $A$ be a real symmetric matrix and $B$ be a
principal submatrix of $A.$ Then the eigenvalues of $B$ interlace the eigenvalues of $A.$
\end{lemma}

Consider an $n\times n$ real symmetric matrix
\[
M=\left(\begin{array}{ccccccc}
M_{1,1}&M_{1,2}&\cdots &M_{1,m}\\
M_{2,1}&M_{2,2}&\cdots &M_{2,m}\\
\vdots& \vdots& \ddots& \vdots\\
M_{m,1}&M_{m,2}&\cdots &M_{m,m}\\
\end{array}\right),
\]
whose rows and columns are partitioned according to a partitioning
$X_{1}, X_{2},\ldots ,X_{m}$ of $\{1,2,\ldots, n\}$. The \emph{quotient matrix}
$R$ of the matrix $M$ is the $m\times m$ matrix whose entries are the
average row sums of the blocks $M_{i,j}$ of $M$. The partition is \emph{equitable}
if each block $M_{i,j}$ of $M$ has constant row (and column) sum.

\begin{lemma}(Brouwer and Haemers \cite{BrHa12, Haem95})\label{le2.2}
Let $M$ be a real symmetric matrix. Then the eigenvalues of every
quotient matrix of $M$ interlace the ones of $M.$ Furthermore, if the
interlacing is tight, then the partition is equitable.
\end{lemma}

\section{Proof of Theorem \ref{main1}}

Following \cite{BoMu08}, for  disjoint subsets $X$ and $Y$ of $V(G)$, let $E(X, Y)$ be the set of edges with
one end in $X$ and the other end in $Y$, and
\[
e(X, Y)=|E(X, Y)|, \mbox{ and } d(X) = e(X, V(G) - X).
\]
Tutte \cite{Tutt47} initiated the cage problem, which seeks,
for any given integers $d$ and $g$ with
$d\geq2$ and $g\geq3$, the smallest possible number $n(d ,g)$
such that there exists a $d$-regular simple graph with girth $g$.
A tight lower bound (often referred as the Moore bound) on $n(d ,g)$ can be found in \cite{ExJa11}.

\begin{lemma}(Exoo and Jajcay \cite{ExJa11})\label{le3.0}
For given integers $d \ge 2$ and $g \ge 3$, let $t = \lfloor \frac{g-1}{2} \rfloor$. Then
\begin{equation*}
n(d ,g)\geq
\left\{
\begin{array}{lc}
1+d\sum_{i=0}^{t-1}(d-1)^{i}, &\, \text{ $g=2t+1$};\\
2\sum_{i=0}^{t}(d-1)^{i}, &\, \text{ $g=2t+2$}.
\end{array}
\right.
\end{equation*}
\end{lemma}

We start our arguments with a technical lemma. For a subset $X \subseteq V(G)$, define
$\overline{X} = V(G) - X$, and
$N_G(X) = \{u \in \overline{X}: \exists ~v \in X$ such that $uv \in E(G)\}$. If $X = \{v\}$, then
we use $N_G(v)$ for $N_G(\{v\})$. When $G$ is understood from the context, we often
omit the subscript $G$.

\begin{lemma}\label{le3.1}
Let $G$ be a simple graph with minimum degree $\delta = \delta(G) \ge 2$ and girth $g = g(G) \ge 3$,
and $X$ be a vertex subset of $G$. Let $n_1^* = n_1^*(\delta, g)$ be defined as in (\ref{n1*}).
If $d(X)<\delta$, then $|X|=n_{1}\geq n_{1}^{*}$.
\end{lemma}

\begin{proof} For notational convenience,
we use $X$ to denote both a vertex subset of $G$ as well
as $G[X]$, the subgraph induced by the vertices of $X$.

\begin{claim} \label{cl-0}
$X$ contains at least a cycle.
\end{claim}

By contradiction, assume that $X$ is acyclic. Then $|E(X)| \le n_1-1$, and so
\[
\delta \cdot n_{1}=\delta \cdot |X|\leq \sum_{v\in X}d_{G}(v)=2|E(X)|+e(X, Y)\leq 2(n_{1}-1)+\delta-1,
\]
leading to a contradiction $n_{1}\leq \frac{\delta-3}{\delta-2}<1$.
This proves Claim \ref{cl-0}.

By Claim \ref{cl-0}, $X$ must contain a cycle with length at least $g$.
We shall justify the lemma by making a sequence of claims.

\begin{claim} \label{cl-1}
Each of the following holds.
\\
(i) If $g \ge 3$, then there exists a vertex $u_0  \in X$ such that $N(u_0) \cap \overline{X}=\emptyset$.
\\
(ii) If $g \ge 3$, then $X$ contains a path $P=u_0u_{1}u_{2}\cdots u_{g-3}$ such that
for any $i \in \{0, 1, 2, ..., g-3\}$, $N(u_i) \cap \overline{X} = \emptyset$,
for the neighborhood of whose each vertex is contained in $X$.
\end{claim}

If (i) does not hold, then for every vertex $v\in X$, we always have $N(v)\cap\overline{X}\neq\emptyset$.
Fix a vertex $v_{0}\in X$. Then
\begin{eqnarray*}
d(X) & = & |N(v_{0})\cap\overline{X}|+|e(X-\{v_{0}\}, \overline{X})|\geq|N(v_{0})\cap\overline{X}|+|X-\{v_{0}\}|
\\
& \ge & |N(v_{0})\cap\overline{X}|+|N(v_{0})\cap X|
=d(v_{0})\geq\delta,
\end{eqnarray*}
contrary to the fact $d(X)<\delta$. Hence (i) follows.

We shall prove (ii) by induction on $g$. By (i), (ii) holds if $g = 3$. Assume that
$g \ge 4$ and (ii) holds for smaller values of $g$. Thus $X$ contains a
path $P' = u_0u_1\cdots u_{g-4}$ such that
for any $i \in \{0, 1, 2, ..., g-4\}$, $N(u_i) \cap \overline{X} = \emptyset$.
Let
$N' = \{u' \in N(u_0): N(u') \cap \overline{X} \neq \emptyset\}$ and
$N''  = \{u'' \in N(u_{g-4}): N(u'') \cap \overline{X} \neq \emptyset\}$.
Since $g(G) = g$, for any $w \in  N(u_{0})$,  $N(w) \cap V(P') = \{u_{0}\}$, and
for any $w \in  N(u_{g-4})$,  $N(w) \cap V(P') = \{u_{g-4}\}$.
As $u_{g-4} \in X$ and $|N(u_{g-4}) - V(P')| \ge \delta - 1 \ge d(X) \ge |N''|$, either
$|N(u_{g-4}) - V(P')| > |N''|$, and so there must be a vertex $u_{g-3} \in N(u_{g-4}) - (V(P') \cup N'')$;
or $|N(u_{g-4}) - V(P')| = |N''|$.
If $|N(u_{g-4}) - V(P')| > |N''|$, then a path $P=u_0u_{1}u_{2}\cdots u_{g-3}$
satisfying (ii) is found, and so (ii) holds by induction in this case.
Hence we assume that $|N(u_{g-4}) - V(P')| = d(X) = |N''|$. This implies
that $N' = \emptyset$ as for any $u' \in N'$, there must be a vertex $w' \in \overline{X}$
such that $u'w' \in E(G)$. Since $d(X) = |N''|$, this forces that $u' \in N''$, and so
$E(P') \cup \{u_0u', u'u_{g-4}\}$ is a cycle of
length $g-2$, contrary to the assumption that the girth of $G$ is $g$.
Hence if $|N(u_{g-4}) - V(P')| = d(X) = |N''|$, then $N' = \emptyset$, and so there must
be a vertex $u_{-1} \in N(u_0) - V(P')$ such that $N(u_{-1}) \cap \overline{X} = \emptyset$.
This implies that, letting $v_i = u_{i-1}$ for $0 \le i \le g-3$, we obtain
a path $P = v_0v_1\cdots v_{g-3}$ such that for any $i \in \{0, 1, 2, ..., g-3\}$, $N(v_i) \cap \overline{X} = \emptyset$.
Hence (ii) is proved by induction.
This justifies the claim.

Let $t = \lfloor \frac{g-1}{2} \rfloor$. By Lemma \ref{le3.0} and by Claim \ref{cl-1}(ii), if $g = 2t+1$ is odd, then
\begin{eqnarray}
|X| & \ge & 1+\delta\sum_{i=0}^{t-1}(\delta-1)^{i}-d(X)-d(X)(\delta-1)-\cdots-d(X)(\delta-1)^{t-2}
\\ \nonumber
& \ge & 1+\delta\sum_{i=0}^{t-1}(\delta-1)^{i}-\sum_{i=1}^{t-1}(\delta-1)^{i}
=1+\delta+\sum_{i=2}^{t}(\delta-1)^{i}=n_{1}^{*}.
\end{eqnarray}
By the same reason, if  $g = 2t+2$ is even, then
\begin{eqnarray}
|X| & \ge & 2\sum_{i=0}^{t}(\delta-1)^{i}-d(X)-d(X)(\delta-1)-\cdots-d(X)(\delta-1)^{t-2}
\\ \nonumber
& \ge & 2\sum_{i=0}^{t}(\delta-1)^{i}-\sum_{i=1}^{t-1}(\delta-1)^{i}
=2+2(\delta-1)^{t}+\sum_{i=1}^{t-1}(\delta-1)^{i}=n_{1}^{*}.
\end{eqnarray}
This completes the proof of the lemma. \hspace*{\fill}$\Box$
\end{proof}

\subsection{Proof of Theorem \ref{main1}(i)}

Suppose that $k$ is an integer with $k \ge 2$. By contradiction, we assume that $ \kappa'(G) = r \leq k-1$.
Then there exists a partition $(X, Y)$ with $Y = \overline{X}$ such that
$e(X, Y)=r\leq k-1\leq \delta-1$. Let $|X|=n_{1}, |Y|=n_{2}$.
By Lemma \ref{le3.1} and as $n_{1}+n_{2}=n$,
we have $n_{1}^{*} \leq \min\{n_{1}, n_2\} \le \frac{n}{2} \le n-n_{1}^{*}$.
Hence $n_{1}n_{2}=n_{1}(n-n_{1})\geq n_{1}^{*}(n-n_{1}^{*})$.

Let $\bar{d_{1}}=\frac{1}{n_{1}}\sum_{v\in X}d(v)$, $\bar{d_{2}}=\frac{1}{n_{2}}\sum_{v\in Y}d(v)$. Then $\bar{d_{1}}, \bar{d_{2}}\geq \delta$.
Accordingly,
the quotient matrix $R(aD+A)$ of $aD+A$ on the partition $(X, Y)$ becomes:
\begin{equation*}
R(aD+A)=
\left(
\begin{array}{cc}
(a+1)\bar{d_{1}}-\frac{r}{n_{1}} &\, \frac{r}{n_{1}}\\
\frac{r}{n_{2}} &\, (a+1)\bar{d_{2}}-\frac{r}{n_{2}}\\
\end{array}
\right).
\end{equation*}
As the characteristic polynomial of $R(aD+A)$ is $$
\lambda^{2}-[(a+1)\bar{d_{1}}-\frac{r}{n_{1}}+(a+1)\bar{d_{2}}-
\frac{r}{n_{2}}]\lambda+[(a+1)\bar{d_{1}}-\frac{r}{n_{1}}][(a+1)\bar{d_{2}}
-\frac{r}{n_{2}}]-\frac{r^{2}}{n_{1}n_{2}},$$
we have, by direct computation,
\begin{eqnarray}
\lambda_{2}(R) &=&\frac{1}{2}\{[(a+1)\bar{d_{1}}-\frac{r}{n_{1}}+
(a+1)\bar{d_{2}}-\frac{r}{n_{2}}]
\\ \nonumber
& \; & -\sqrt{[(a+1)\bar{d_{1}}-\frac{r}{n_{1}}+(a+1)\bar{d_{2}}-\frac{r}{n_{2}}]^{2}
-4[(a+1)\bar{d_{1}}-\frac{r}{n_{1}}][(a+1)\bar{d_{2}}-\frac{r}{n_{2}}]+\frac{4r^{2}}{n_{1}n_{2}}} \}
\\ \nonumber
&=&\frac{1}{2}\{[(a+1)\bar{d_{1}}-\frac{r}{n_{1}}+(a+1)\bar{d_{2}}-\frac{r}{n_{2}}]-
\sqrt{[(a+1)\bar{d_{1}}-\frac{r}{n_{1}}-(a+1)\bar{d_{2}}+\frac{r}{n_{2}}]^{2}+\frac{4r^{2}}{n_{1}n_{2}}}\}
\\ \nonumber
&=&\frac{1}{2}\{[(a+1)\bar{d_{1}}-\frac{r}{n_{1}}+(a+1)\bar{d_{2}}-\frac{r}{n_{2}}]-
\sqrt{[(a+1)(\bar{d_{1}}-\bar{d_{2}})-(\frac{r}{n_{1}}-\frac{r}{n_{2}})]^{2}+\frac{4r^{2}}{n_{1}n_{2}}}\}
\\ \nonumber
& =&\frac{1}{2}\{[(a+1)\bar{d_{1}}-\frac{r}{n_{1}}+(a+1)\bar{d_{2}}-\frac{r}{n_{2}}]
\\ \nonumber
&\;& -\sqrt{(a+1)^{2}(\bar{d_{1}}-\bar{d_{2}})^{2}+(\frac{r}{n_{1}}-\frac{r}{n_{2}})^{2}-2(a+1)
(\bar{d_{1}}-\bar{d_{2}})(\frac{r}{n_{1}}-\frac{r}{n_{2}})+\frac{4r^{2}}{n_{1}n_{2}}}\}
\\ \nonumber
&=&\frac{1}{2}\{[(a+1)(\bar{d_{1}}+\bar{d_{2}})-\frac{r}{n_{1}}-\frac{r}{n_{2}}]
\\ \nonumber
&\; &- \sqrt{(a+1)^{2}(\bar{d_{1}}-\bar{d_{2}})^{2}+(\frac{r}{n_{1}}+\frac{r}{n_{2}})^{2}+
2(a+1)(\bar{d_{1}}-\bar{d_{2}})(\frac{r}{n_{2}}-\frac{r}{n_{1}})}\}
\\ \nonumber
&\geq &\frac{1}{2}\{[(a+1)(\bar{d_{1}}+\bar{d_{2}})-\frac{r}{n_{1}}-\frac{r}{n_{2}}]
\\ \nonumber
&\;& -\sqrt{(a+1)^{2}(\bar{d_{1}}-\bar{d_{2}})^{2}+(\frac{r}{n_{1}}+\frac{r}{n_{2}})^{2}+
2(a+1)|\bar{d_{1}}-\bar{d_{2}}|(\frac{r}{n_{1}}+\frac{r}{n_{2}})}\}
\\ \nonumber
&=&\frac{1}{2}\{[(a+1)(\bar{d_{1}}+\bar{d_{2}})-\frac{r}{n_{1}}-\frac{r}{n_{2}}]-
[(a+1)|\bar{d_{1}}-\bar{d_{2}}|+(\frac{r}{n_{1}}+\frac{r}{n_{2}})]\}
\\ \nonumber
&=&\min[(a+1)\bar{d_{1}},(a+1)\bar{d_{2}}]-\frac{rn}{n_{1}n_{2}}
\\ \label{L2R}
&\geq&(a+1)\delta-\frac{(k-1)n}{n_{1}^{*}(n-n_{1}^{*})}.
\end{eqnarray}
By Lemma \ref{le2.2},
$\lambda_{2}(G, a)\geq\lambda_{2}(R)\geq(a+1)\delta-\frac{(k-1)n}{n_{1}^{*}(n-n_{1}^{*})}$.
By assumption, $\lambda_{2}(G, a)\leq (a+1)\delta-\frac{(k-1)n}{n_{1}^{*}(n-n_{1}^{*})}$,
and so we must have $\lambda_{2}(G, a)=\lambda_{2}(R)=(a+1)\delta-\frac{(k-1)n}{n_{1}^{*}(n-n_{1}^{*})}$.
It follows that all the inequalities in (\ref{L2R}) must be equalities.
Hence $r=k-1$ and $\bar{d_{1}}=\bar{d_{2}}=\delta$, implying that $G$ must be a $\delta$-regular graph,
and so $\lambda_{1}(G, a)=(a+1)\delta$. By algebraic manipulation,

\begin{equation*}
\begin{split}
\lambda_{1}(R)=&\frac{1}{2}\{[(a+1)\delta-\frac{r}{n_{1}}+(a+1)\delta-\frac{r}{n_{2}}]\\
+&\sqrt{[(a+1)\delta-\frac{r}{n_{1}}+(a+1)\delta-\frac{r}{n_{2}}]^{2}-4[(a+1)\delta
-\frac{r}{n_{1}}][(a+1)\delta-\frac{r}{n_{2}}]+\frac{4r^{2}}{n_{1}n_{2}}}\}\\
=&\frac{1}{2}\{[2(a+1)\delta-\frac{r}{n_{1}}-\frac{r}{n_{2}}]+\sqrt{[(a+1)\delta
-\frac{r}{n_{1}}-((a+1)\delta-\frac{r}{n_{2}})]^{2}+\frac{4r^{2}}{n_{1}n_{2}}}\}\\
=&\frac{1}{2}\{[2(a+1)\delta-\frac{r}{n_{1}}-\frac{r}{n_{2}}]+\sqrt{(\frac{r}{n_{1}}
-\frac{r}{n_{2}})^{2}+\frac{4r^{2}}{n_{1}n_{2}}}\}\\
=&\frac{1}{2}\{[2(a+1)\delta-\frac{r}{n_{1}}-\frac{r}{n_{2}}]+(\frac{r}{n_{1}}+\frac{r}{n_{2}})\}\\
=&(a+1)\delta.
\end{split}
\end{equation*}
Therefore, the interlacing is tight. By Lemma \ref{le2.2}, the partition is equitable.
This means that every vertex in $X$ has the same number of neighbors in $Y$.
However, by Claim \ref{cl-1}(i) of Lemma \ref{le3.1},
there exists at least one  vertex in $X$ without a neighbor in $Y$.
This implies that $r=e(X, Y)=k-1=0$, contrary to the assumption that $k \ge 2$. \hspace*{\fill}$\Box$

\subsection{Corollaries of Theorem \ref{main1}(i)}

Throughout this subsection, $n_1^*$ is defined as in (\ref{n1*}).
To see that Theorem \ref{main1}(ii) follows from Theorem \ref{main1}(i), we
observe that as
$n_{1}^{*} \leq \min\{n_{1}, n_2\} \le \frac{n}{2} \le n-n_{1}^{*}$, it follows that
\begin{equation} \label{2k}
(a+1)\delta-\frac{2(k-1)}{n_{1}^{*}} \le
(a+1)\delta-\frac{(k-1)n}{n_{1}^{*}(n-n_{1}^{*})},
\end{equation}
and so Theorem \ref{main1}(ii) follows from Theorem \ref{main1}(i).

For real numbers $a$ and $b$ with $\frac{a}{b}\geq-1$, let $\lambda_{i}(G, a, b)$ be
the $i$th largest eigenvalues of the matrix $aD+bA$.
Thus $\lambda_{i}(G, a, 1)=\lambda_{i}(G, a)$.

\begin{corollary}\label{co3.2}
Let $a$ and $b$ be real numbers with with $b \neq 0$ and  $\frac{a}{b}\geq-1$,
$k$ be an integer with $k\geq 2$, and $G$ be a simple graph with $n = |V(G)|$, $g =g(G)$ and with minimum degree
$\delta = \delta(G) \geq k$. Then $\kappa'(G)\geq k$ if one of the following holds.
\\
(i) $b>0$ and $\lambda_{2}(G, a, b)\leq (a+b)\delta-\frac{b(k-1)n}{n_{1}^{*}(n-n_{1}^{*})}$.
\\
(ii) $b < 0$ and $\lambda_{n-1}(G, a, b)\geq (a+b)\delta-\frac{b(k-1)n}{n_{1}^{*}(n-n_{1}^{*})}$.
\end{corollary}

\begin{proof}
As $aD+bA=b(\frac{a}{b}D+A)$, it follows by definition
that
\begin{equation} \label{ab}
\left\{
\begin{array}{ll}
\mbox{ if $b> 0$, } & \mbox{ then $\lambda_{i}(G, a, b)=b\lambda_{i}(G, \frac{a}{b})$; and}
\\
\mbox{ if $b< 0$, } & \mbox{ then $\lambda_{n-i+1}(G, a, b)=b\lambda_{i}(G, \frac{a}{b})$.}
\end{array} \right.
\end{equation}
Hence Corollary \ref{co3.2} follows form Theorem \ref{main1}(i). \hspace*{\fill}$\Box$
\end{proof}

Choosing $a \in \{0, -1, 1\}$ and $b=1$ in Corollary \ref{co3.2}, we have the following special case.

\begin{corollary}\label{co3.3}
Let $k$ be an integer with $k\geq 2$, and $G$ be a simple graph with $n = |V(G)|$, $g = g(G)$ and with minimum degree
$\delta = \delta(G) \geq k$. Each of the following holds.
\\
(i) If $\lambda_{2}(G)\leq \delta-\frac{(k-1)n}{n_{1}^{*}(n-n_{1}^{*})}$, then $\kappa'(G)\geq k$.
\\
(ii) If $\mu_{n-1}(G)\geq \frac{(k-1)n}{n_{1}^{*}(n-n_{1}^{*})}$, then $\kappa'(G)\geq k$.
\\
(iii) If $q_{2}(G)\leq 2\delta-\frac{(k-1)n}{n_{1}^{*}(n-n_{1}^{*})}$, then $\kappa'(G)\geq k$.
\end{corollary}

As $n^*_1(\delta, 3) = \delta+1$ and by (\ref{2k}), Theorem \ref{edge-conn} (iii) and (iv) are consequences
of Corollary \ref{co3.3}. Corollary \ref{co3.3} also implies the following result on bipartite graphs by setting $g \ge 4$
in  Corollary \ref{co3.3}.

\begin{corollary}\label{co3.5}
Let $G$ be a bipartite graph with minimum degree $\delta \geq k\geq 2$.
If $\lambda_{2}(G)<\delta-\frac{k-1}{\delta}$, then $\kappa'(G)\geq k$.
\end{corollary}

\section{Proof of Theorem \ref{main2} and its Corollaries}

Throughout this section, for given integers $\delta$ and $g$, we continue
defining $n_1^* = n_1^*(\delta, g)$ as in (\ref{n1*}).
We utilize the arguments deployed in \cite{LHGL14} to prove
Theorem \ref{main2} by imposing the girth requirement.
In particular, the following technical
lemma will also be used, with an additional condition
$a \ge -1$ to justify the algebraic manipulation needed in
the proof of the lemma.

\begin{lemma} (Lemma 3.2 of \cite{LHGL14}) \label{le4.1}
Let $a\geq-1$ be a real number and $G$ be a simple graph with
minimum degree $\delta = \delta(G)$. For any two disjoint nonempty vertex  subsets $X$ and $Y$,
if $\lambda_{2}(G, a)\leq (a+1)\delta-\max\{\frac{d(X)}{|X|}, \frac{d(Y)}{|Y|}\}$, then
\[
[e(X, Y)]^{2}\geq [(a+1)\delta-\frac{d(X)}{|X|}-\lambda_{2}(G, a)][(a+1)\delta-\frac{d(Y)}{|Y|}-\lambda_{2}(G, a)]|X||Y|.
\]
\end{lemma}

\noindent {\bf Proof of Theorem \ref{main2}. }
Let $V_{1}, \ldots, V_{t}$ be an arbitrary partition of $V(G)$. Without loss of generality,
we assume that  $d(V_{1})\leq d(V_{2})\leq \cdots \leq d(V_{t})$.
By Theorem \ref{tree-packing}, it suffices to show that
$\sum_{i=1}^t d(V_{i}) \ge 2k(t-1)$. The inequality holds trivially if $t =1$. Hence we
assume that $t \ge 2$.
If $d(V_{1})\geq 2k$, then $\sum_{i=1}^t d(V_{i}) \ge t(2k) > 2k(t-1)$.
Thus we also assume that $d(V_{1})\leq 2k-1$.

Let $s$ be the largest integer such that $d(V_{s}) \leq 2k-1$. Then as $d(V_{1})\leq 2k-1$,
$1 \le s \le t$, and if $s < t$, then $d(V_{s+1})\geq 2k$.
By Lemma \ref{le3.1}, $|V_{i}|\geq n_{1}^{*}$ for $1\leq i\leq s$.
It follows that for any $i$ with $i\leq s$,
\begin{equation} \label{L2}
\lambda_{2}(G, a)<(a+1)\delta-\frac{2k-1}{n_{1}^{*}}\leq(a+1)\delta
-\max\{\frac{d(V_{1})}{|V_{1}|}, \frac{d(V_{i})}{|V_{i}|}\}.
\end{equation}
By (\ref{L2}) and Lemma \ref{le4.1},
\begin{equation*}
\begin{split}
[e(V_{1}, V_{i})]^{2}\geq&[(a+1)\delta-\frac{d(V_{1})}{|V_{1}|}-\lambda_{2}(G, a)][(a+1)\delta-\frac{d(V_{i})}{|V_{i}|}-\lambda_{2}(G, a)]|V_{1}|\cdot|V_{i}|\\
>&[\frac{2k-1}{n_{1}^{*}}-\frac{d(V_{1})}{|V_{1}|}]|V_{1}|[\frac{2k-1}{n_{1}^{*}}-\frac{d(V_{i})}{|V_{i}|}]|V_{i}|\\
\geq&[2k-1-d(V_{1})][2k-1-d(V_{i})]\\
\geq&[2k-1-d(V_{i})]^{2}.
\end{split}
\end{equation*}
Hence $e(V_{1}, V_{i})>2k-1-d(V_{i})$, or $e(V_{1}, V_{i})\geq 2k-d(V_{i})$.
It follows that $\sum_{i=2}^{s}e(V_{1}, V_{i})\geq \sum_{i=2}^{s}(2k-d(V_{i}))$, and so
as $d(V_j) \ge 2k$ for all $j \ge s+1$, we have
\begin{eqnarray}
\sum_{i=1}^t d(V_{i}) & = & d(V_1) + \sum_{i=2}^s d(V_{i}) + \sum_{i=s+1}^t d(V_{i})
\\ \nonumber
& \ge & \sum_{i=2}^s e(V_{1}, V_{i}) + \sum_{i=2}^s d(V_{i}) + \sum_{i=s+1}^t d(V_{i})
\\ \nonumber
& \ge & 2k(s-1) - \sum_{i=2}^s d(V_{i}) + \sum_{i=2}^s d(V_{i}) + \sum_{i=s+1}^t d(V_{i})
\\
& \ge & 2k(s-1) + 2k(t-s) = 2k(t-1).
\end{eqnarray}
Hence by Theorem \ref{tree-packing}, $\tau(G) \ge k$, as desired.
This completes the proof of Theorem \ref{main2}.

The following seemingly more general result
can be derived from Theorem \ref{main2} by arguing similarly as in \cite{LHGL14} and using (\ref{ab}),
within certain ranges of the real numbers $a$ and $b$.

\begin{corollary}\label{th4.3}
Let $a$ and $b$ be real numbers satisfying $b \neq 0$ and $\frac{a}{b} \ge -1$,
$k$ be an integer with $k > 0$ and $G$ be a graph with $n = |V(G)|$,  $g = g(G)$
and with minimum degree $\delta = \delta(G) \geq2k$. Each of the following holds.
\\
(i)
If $b>0$ and $\lambda_{2}(G, a, b)< (a+b)\delta-\frac{b(2k-1)}{n_{1}^{*}}$, then $\tau(G)\geq k$.
\\
(ii)
If $b<0$ and $\lambda_{n-1}(G, a, b)> (a+b)\delta-\frac{b(2k-1)}{n_{1}^{*}}$, then $\tau(G)\geq k$.
\end{corollary}

Thus Corollary \ref{cor2} now follows by letting $a \in \{0, 1, -1\}$ and $b = 1$ in Corollary \ref{th4.3}.

\noindent
{\bf Acknowledgement. } The research of R. Liu is partially
supported by National Natural Science Foundation of China (No.~11571323), Outstanding Young Talent Research Fund of Zhengzhou University (No.~1521315002), the China Postdoctoral Science Foundation (No.~2017M612410) and Foundation for University Key Teacher of Henan Province (No.~2016GGJS-007).
The research of Hong-Jian Lai is partially supported by National Natural Science
Foundation of China grants CNNSF 11771039 and CNNSF 11771443. The research of Y. Tian is partially supported by National Natural Science
Foundation of China grants CNNSF 11531011.

\small {

}

\end{document}